\documentclass[reqno]{amsart}
\usepackage{amsfonts}
\usepackage{amsfonts}
\usepackage{amsfonts}
\usepackage{amsfonts}
\usepackage{amsfonts}
\usepackage{amsfonts}
\usepackage{}
\usepackage{amsfonts}
\usepackage{bbm}
\usepackage{mathrsfs}
\usepackage{mathrsfs}
\usepackage{mathrsfs}
\usepackage{bbm}
\usepackage{mathrsfs}
\usepackage{mathrsfs}
\usepackage{bbm}
\usepackage{amsfonts,amssymb,amsmath,amsthm}
\usepackage{mathrsfs}
\usepackage{graphicx}
\usepackage{url}
\usepackage{enumerate}

\newtheorem{theorem}{Theorem}[section]
\newtheorem{lemma}[theorem]{Lemma}
\newtheorem{proposition}[theorem]{Proposition}
\newtheorem{corollary}[theorem]{Corollary}
\theoremstyle{definition}
\newtheorem{definition}[theorem]{Definition}
\newtheorem{remark}{Remark}
\newtheorem{example}{Example}

\numberwithin{equation}{section}

\DeclareMathOperator{\vol}{Vol}

\author{Wei Zhao}
\address{
Department of Mathematics\\
East China University of Science and Technology\\
Shanghai, China}
\email{szhao\underline{ }wei@yahoo.com}

\begin{document}

\title[]{A Gauss-Bonnet-Chern theorem for complex Finsler manifolds}

\begin{abstract}
In this paper, we establish a Gauss-Bonnet-Chern theorem for general closed complex Finsler manifolds.
\end{abstract}
\maketitle

\section{Introduction}
One of the most important aspects of differential geometry is that which investigates the relationship between the curvature properties and the topology structure of a manifold. Among all the significant results in this aspect, a beautiful one is the Gauss-Bonnet-Chern theorem, which is established by S. S. Chern in \cite{chg,chg-2}.
Explicitly, given an arbitrary
oriented closed $n$-dimensional Riemannian manifolds $(M,g)$, one has
\[
\int_Me(TM)=\chi(M),\tag{1.1}\label{1.1}
\]
where $e(TM)$ is the (geometric) Euler form of $TM$ induced by the Levi-Civita connection, i.e.,
\begin{align*}
e(TM)&=\left\{
\begin{array}{lll}
&\frac{(-1)^{p}}{2^{2p}\pi^pp!}\epsilon_{i_1\ldots i_{2p}}\Omega_{i_1}^{i_2}\wedge\cdots\wedge \Omega_{i_{2p-1}}^{i_{2p}},&n=2p,\\
\\
&0,&n=2p+1.
\end{array}
\right.
\end{align*}

To carry matters further, let $(M,h)$ now be a closed $n$-dimensional Hermitian
manifold and let $c_n(T'M)$ denote the ($n$-th) Chern form of $T'M$ induced by the canonical Hermitian connection, that is,
\[
c_n(T'M)=\frac{1}{n!}\left(\frac{\sqrt{-1}}{2\pi}\right)^n\delta^{i_1\cdots i_n}_{j_1\cdots j_n}\Omega^{j_1}_{i_1}\wedge\cdots \wedge \Omega^{j_n}_{i_n}.
\]
Bott and Chern in \cite{BC} show that the Gauss-Bonnet-Chern theorem in the present context becomes
\[
\int_M c_n(T'M)=\chi(M).\tag{1.2}\label{1.3}
\]

It should be remarked that both Riemannian and Hermitian metrics are quadratic metrics, and hence, the Euler forms (resp. the Chern forms) of such metrics determine a characteristic class on the underlying manifold $M$, which is the so-called Euler class of $TM$ (resp. Chern class of $T'M$). Stokes' formula then furnishes the characteristic class versions of (\ref{1.1}) and (\ref{1.3}). Refer to \cite{BC,ch-c1,MS} for more details.

Finsler geometry is metric differential geometry without quadratic restriction. Hence, it is entirely natural to ask whether an analogue of the Gauss-Bonnet-Chern theorem still holds for Finsler manifolds.

We first consider the case of real Finsler manifolds. Suppose that $(M,F)$ is a real $n$-dimensional oriented closed Finsler manifold. Let $\pi:SM\rightarrow M$ denote the projective sphere bundle and let $\pi^*TM$ denote the pull-back tangent bundle. The Finsler metric $F$ induces a Riemannian metric on $\pi^*TM$. There are many important Finslerian connections on $\pi^*TM$, but none of them is both torsion-free and metric-compatible (cf.\,\cite{BCS}). In particular, the Euler forms induced by such connections cannot live on $M$ but on $SM$. Hence, there is no simple formula as (\ref{1.1}) valid for general real Finsler manifolds.

It is D. Bao and S. S. Chern \cite{BS-1} who first point out the significance of $\vol(x)$, the Riemannian volume of $S_xM:=\pi^{-1}(x)$ induced by $F$. By the Chern connection, they establish the Gauss-Bonnet-Chern theorem for real Finsler manifolds with $\vol(x)=\text{const}$.
Since then, efforts have been made and the Gauss-Bonnet-Chern theorem for general real Finsler manifolds has been established recently (cf.\,\cite{Dao,L,Sh2,Sh3,Z}). Explicitly,
given any torsion-free or metric-compatible connection $D$, we have
\[
\int_{M}\frac{[X]^*e(\pi^*TM;D)+\mathfrak{A}_F(X,D)}{\vol(x)}=\frac{\chi(M)}{\vol(\mathbb{S}^{n-1})},\tag{1.3}\label{1.4}
\]
where $[X]:M\rightarrow SM$ is an arbitrary section with isolated singularities, $e(\pi^*TM;D)$ is the Euler form induced by $D$, and $\mathfrak{A}_F(X,D)$ is a current on $M$  depending only  on $X$, $D$ and $F$. The corrected item $\mathfrak{A}_F(X,D)$ is introduced, because $\vol(x)$ is not a constant generally and $D$ cannot be both torsion-free and metric-compatible. In the Riemannian case, $\vol(x)=\vol(\mathbb{S}^{n-1})$ and $\int_{M}\mathfrak{A}_F(X,D)=0$ (cf.\, \cite{L,Z}). Hence,
(\ref{1.4}) implies (\ref{1.1}).

The Gauss-Bonnet-Chern theorem has been well developed on real Finsler manifolds, while the complex Finsler case has not been studied at the same space. Little work has been made to this subject. The purpose of this paper is to establish a Gauss-Bonnet-Chern theorem for complex Finsler manifolds.

The complex Finsler case is much different from the real one. For example, in view of (\ref{1.4}), it seems natural to investigate the Chern form of  the pull-back bundle $\mathfrak{p}^*T'M$, which is induced by the projective holomorphic tangent bundle $\mathfrak{p}:PM=T'M/(\mathbb{C}-\{0\})\rightarrow M$.
However, since the (complex) rank of $PM$ is $(n-1)$, the Bott-Chern form \cite{BC} implies that the integral of the pull-back Chern form of $\mathfrak{p}^*T'M$ over $M$ cannot produce the Euler characteristic of $M$.

 For this reason, we turn to the pull-back bundle $\pi^*T'M$, where $\pi:T'M\backslash0\rightarrow M$ is the splitting holomorphic bundle. Since the Finsler metric induces a Hermitian metric on $\pi^*T'M$, by the transgression formula \cite[(2.19)]{BC}, we just need to investigate the Chern form induced by  the canonical Hermitian connection $\nabla$ on $\pi^*T'M$, which is denoted by $c_n(\pi^*T'M;\nabla)$. Let $X$ be a holomorphic vector field on $M$ with non-degenerate zeros. The main difficulty in establishing the Gauss-Bonnet-Chern theorem is to modify $X^*c_n(\pi^*T'M;\nabla)$ such that the integral of the modification over $M$ produces the Euler characteristic of $M$.
 In view of the methods in the real Finlserian case \cite{BS-1,L,Sh3,Z},  one might use
 the Bott-Chern form \cite{BC} to obtain
 \[
c_n(\pi^*T'M;\nabla)=\frac{d\left({d^c\eta}\right)}{4\pi},\tag{1.4}\label{1.5}
 \]
 where $d^c=\sqrt{-1}(\bar{\partial}-\partial)$ and $\eta$ is the Bott-Chern form on $T'M\backslash0$. But the current $X^*d^c\eta$ is rather abstract and therefore, difficult to deal with. In the Hermitian case,  without knowing the explicit formula of $\eta$, one can still  calculate the value of $\int_Md(X^*d^c\eta)$  by constructing a trivial bundle (see \cite{BC,C}). However, this method cannot be applied to the Finsler setting, since the structure of a Finsler metric is much more complicated than the one of a Hermitian metric.

Using Chern's transgression method \cite{C}, we construct a current $\Psi$  on $M$ directly such that
 \[
 X^*c_n(\pi^*T'M;\nabla)=-d\Psi.\tag{1.5}\label{1.6}
 \]
Unlike the case of the Bott-Chern form, we know the explicit formula of $\Psi$. If $F$ is Hermitian, then the Gauss-Bonnet-Chern theorem follows from the Bochner-Martinelli formula immediately. Unfortunately, this formula is invalid in the Finsler setting and moreover, $-\int_Md\Psi\neq \chi(M)$ in general. Hence, for our purpose, (\ref{1.6}) needs to be modified.
Inspired by Bao-Chern's point of view \cite{BS},
we turn to investigate the complex indicatrix
 \[
 S_zM:=\{\xi\in T'_zM:\, F(z,\xi)=1\},
 \]
 Since the Finsler metric induces a Hermitian metric on each tangent space $T'_zM$, we can define a natural Hermitian volume of $S_zM$,
which is denote by $\vol(z)$. In particular, if the Finsler metric is Hermitian, then $\vol(z)=\vol(\mathbb{S}^{2n-1})$. By a careful argument, we find there is a close relationship between $\vol(z)$ and $\Psi$.  Using $\vol(z)$ to modify (\ref{1.6}), we obtain the following theorem.
\begin{theorem}\label{Th1} Let $(M,F)$ be a closed complex Finsler $n$-manifold and let $X$ be a holomorphic vector field on $M$ with non-degenerate zeros. Given an arbitrary Finslerian connection $D$, we have
\[
\int_M\frac{X^*c_n(\pi^*T'M;D)+\mathfrak{F}_F(X,D)}{\vol(z)}=\frac{\chi(M)}{\vol(\mathbb{S}^{2n-1})},
\]
where $\mathfrak{F}_F(X,D)$ is a current on $M$  depending only on $X$, $D$ and $F$.
\end{theorem}

It is conceivable that for some special complex Finsler manifolds, Theorem \ref{Th1} admits significant simplification.
For example, see Example \ref{realex} and Corollary \ref{Bergs} below. In particular, for the Hermitian case, one can show $\int_M\mathfrak{F}_F(X,D)=0$. Theorem 1.1 then reduces to (\ref{1.3}). But for non-Hermitian Finsler case,
 $\int_M\frac{\mathfrak{F}_F(X,D)}{\vol(z)}$ does not vanish in general even if $\vol(z)=\text{const}$, because the connection $1$-forms of $D$ live on $T'M\backslash0$. See Section 5 below for more details.

\proof[Acknowledgements]The author wishes to thank Professor Z. Shen and Y. B. Shen for their advice and encouragement.

\section{Preliminaries}
In this section, we recall some definitions and properties concerned with complex Finsler manifolds. See \cite{A,BS,K,SZ} for more details.

Let $M$ be a complex $n$-dimensional manifold $M$ and let $T'M$ be the holomorphic tangent bundle of $M$. A (complex) Finsler metric on $M$ is a continuous function $F:T'M\rightarrow [0,+\infty)$ satisfying the following conditions

(1) $F(\xi) \geq 0$, and $F(\xi) = 0$ if and only if $\xi = 0$;

(2) $F(\lambda \xi) = |\lambda|F(\xi)$ for any $\lambda \in \mathbb{C}\backslash\{ 0\}$;

(3) $F(\xi)$ is smooth outside of the zero-section.

\noindent The pair $(M,F)$ is called a {\it complex Finsler manifold}.
Let $(z^i, \xi^j)=\xi^j\frac{\partial}{\partial z^j}|_{z}$ be a local coordinate system of $T'M$.
Set
\[
G:=F^2, \ G_{*i}:=\frac{\partial}{\partial\xi^i} G_*,\ \ G_{*\bar{i}}:=\frac{\partial}{\partial\bar{\xi}^i} G_*,\ \ G^{i\bar{j}}:=(G_{i\bar{j}})^{-1}.
\]
$F$ is called {\it strongly pseudo-convex} if $(G_{i\bar{j}})$ is positive-definite at every point of $T'M\backslash0$. In this paper, we always assume that $F$ is strongly pseudo-convex.

Note that $G=G(\xi,\bar{\xi})$ is homogenous of degree $1$ in $\xi$ and $\bar{\xi}$. Hence, by Euler's theorem, one can easily show the following proposition.
\begin{proposition}[See \cite{SZ}]\label{cang}
\begin{align*}
\overline{G_{i\bar{j}}}=G_{j\bar{i}},  && G^{i\bar{j}}G_{k\bar{j}}=\delta^i_k,&& G^{i\bar{j}}G_{i\bar{k}}=\delta^{\bar{j}}_{\bar{k}},&& G=G_{i\bar{j}}\xi^i\bar{\xi}^j,\\
G_i\xi^i=G,&& G_{i\bar{j}}\bar{\xi}^j=G_i, && G_{ij}\xi^j=0, && G_{\bar{i}\bar{j}}\bar{\xi}^j=0,\\
G_{ijk}\xi^k=-G_{ij}, && G_{ij\bar{k}}\bar{\xi}^k=G_{ij}, &&G_{i\bar{j}k}\xi^k=0.
\end{align*}
\end{proposition}

Let $\pi:T'M\backslash0\rightarrow M$ be the natural projection and let $\pi^*T'M$ be the pull-back bundle. The Finsler metric $F$ then induces a natural Hermitian metric on $\pi^*T'M$, which is defined by
\[
g_{(z,\xi)}:=G_{i\bar{j}}(z,\xi)\,d\mathfrak{z}^i\otimes d\bar{\mathfrak{z}}^j,
\]
where $\{d\mathfrak{z}^i\}$ is the dual frame field of $\{\frac{\partial}{\partial \mathfrak{z}^i}\}$ and $\frac{\partial}{\partial \mathfrak{z}^i}:=\pi^*\frac{\partial}{\partial z^i}=(z,\xi,\frac{\partial}{\partial z^i}|_z)$, $i=1,\cdots,n$ is the local holomorphic frame for $\pi^*T'M$.

Before considering the connections on $\pi^*T'M$, we first introduce the canonical horizontal splitting of $T'(T'M\backslash0)$, the holomorphic tangent bundle of $T'M\backslash0$.

Denote by $\mathcal {V}$ the vertical subbundle of $T'(T'M\backslash0)$. Thus,
 each Ehresmannn connection $\vartheta\in \mathscr{A}^{1,0}(T'M\backslash0,\mathcal {V})$ defines a horizontal splitting
\[
T'(T'M\backslash0)=\mathcal {H}\oplus \mathcal {V},
\]
where
\[
\mathcal {H}=\text{Span}_\mathbb{C}\left\{\frac{\delta}{\delta z^j}:=\frac{\partial}{\partial z^j}-N_j^i\frac{\partial}{\partial \xi^i}\right\},\ \mathcal {V}=\text{Span}_\mathbb{C}\left\{\frac{\partial}{\partial \xi^j}\right\},
\]
and $N_j^i:=\langle \vartheta(\frac{\partial}{\partial z^j}),\,d\xi^i\rangle$.
Clearly, the dual frame field of $\{\frac{\delta}{\delta z^i},\frac{\partial}{\partial \xi^j}\}$ is
\begin{align*}
\left\{
\begin{array}{lll}
&dz^j,\\
&{\delta \xi^i}:={d\xi^i+N^i_jdz^j}.
\end{array}
\right.
\end{align*}
The exterior differential $d=\partial+\bar{\partial}:\mathscr{A}^\cdot(T'M\backslash0)\rightarrow \mathscr{A}^{\cdot+1}(T'M\backslash0)$ induces a closed real $(1,1)$-form
\[
\omega=\sqrt{-1}\cdot\partial\bar{\partial}G
\]
 on $T'M\backslash0$ such that the restriction $\omega|_z$ on each $\pi^{-1}(z)=T'_zM-\{0\}$
defines a K\"ahler metric and $\omega$ defines a Hermitian metric on
$\mathcal {V}$. It follows from \cite{A} that
there exists a canonical Ehresmann connection such that $\mathcal {H}$ and $\mathcal {V}$ are orthogonal with respect to $\omega$. A simple calculation yields
\[
N^j_k=G^{j\bar{i}}\frac{\partial^2 G}{\partial z^k \partial \bar{\xi}^i}=G^{j\bar{i}}\frac{G_{s\bar{i}}}{\partial z^k}\xi^s.\tag{2.1}\label{Nformula}
\]
In this paper, we always use this canonical horizontal splitting. Clearly, the tangent bundle of $T'M\backslash0$ can be decomposed as
\[
T_{\mathbb{C}}(T'M\backslash0)=(\mathcal {H}\oplus\overline{\mathcal {H}})\oplus(\mathcal {V}\oplus\overline{\mathcal {V}})=:H\oplus V,\tag{2.2}\label{2.1}
\]
which induces a decomposition of $d=\partial+\bar{\partial}:\mathscr{A}^\cdot(T'M\backslash0)\rightarrow \mathscr{A}^{\cdot+1}(T'M\backslash0)$:
\[
d=d^H+d^V=\partial^H+\bar{\partial}^H+\partial^V+\bar{\partial}^V,
\]
where
\begin{align*}
\partial^H:=\frac{\delta}{\delta z^A}dz^A,\
\partial^V:=\frac{\partial}{\partial \xi^i} \delta\xi^i,\ \bar{\partial}^H:=\overline{\partial^H},\ \bar{\partial}^V:=\overline{\partial^V}.
\end{align*}

Now we recall some properties of the connections on $\pi^*T'M$ and introduce the Chern-Finsler connection. See \cite{SZ} for other important Finslerian connections on $\pi^*T'M$.

Given a connection $D$ on $\pi^*T'M$, let $(\theta^i_j)$ denote the connection $1$-forms of $D$ with respect to $\{{\partial}/{\partial \mathfrak{z}^i}\}$. By the horizontal splitting (\ref{2.1}), one can decompose $\theta_s^k$ as
\[
\theta_s^k={\theta^H}_s^k+{\theta^V}_s^k.
\]
It is not hard to see that ${\theta^V}_s^k\otimes\frac{\partial}{\partial \mathfrak{z}^k}\otimes d\mathfrak{z}^s$ is a local section of $\mathscr{A}^1(T'M\backslash0, \pi^*T'M\otimes \pi^*T'^{*}M)$, while $({\theta^H}_s^k)$ defines a connection $D^H$ on $\pi^*T'M$.
In particular, if $D$ is metric-compatible with $g$, then $D^H$ is {\it partially metric-compatible} with $g$, that is,
\[
d^H g(X,Y)=g(D^HX,Y)+g(X,D^HY),\ \forall\, X,Y\in \Gamma(\pi^*T'M).
\]

Note that the $(\pi^*T'M,g)$ is a holomorphic Hermitian bundle. The standard argument (see \cite{BC}) then yields the following theorem. Also refer to \cite{A,K,SZ}.
\begin{theorem}\label{C-F}
There exists a unique connection $\nabla$ on $\pi^*T'M$, called the Chern-Finsler connection, satisfying the following two conditions:

 (1) The connection $1$-form of $\nabla$ is of type $(1,0)$;

 (2) $\nabla$ is metric-compatible with $g$.

\noindent Denote by $( \varpi_i^j)$ the connection $1$-forms of $\nabla$ with respect to $\{\frac{\partial}{\partial \mathfrak{z}^j}\}$. Thus,
\[
\varpi^k_i=\partial G_{i\bar{j}}\cdot G^{k\bar{j}}=G^{k\bar{j}}\frac{\delta G_{i\bar{j}}}{\delta z^s}dz^s+G^{k\bar{j}}G_{i\bar{j}l}\delta \xi^l.
\]
In particular, $(\Gamma_{i,s}^kdz^s):=(G^{k\bar{j}}\frac{\delta G_{i\bar{j}}}{\delta z^s}dz^s)$ also defines a partially metric-compatible connection $\nabla^H$ on $\pi^*T'M$, which is called the Rund connection.
\end{theorem}

According to \cite{A2}, $F$ is called a {\it complex Berwald metric} if $\Gamma_{i,s}^k(z,\xi)=\Gamma_{i,s}^k(z)$. In particular, each connected complex
Berwald manifold is a manifold modeled on a complex Minkowski space. Refer to \cite{A2,BS,SZ} for more details about Berwald manifolds.

The  vertical component of $(\varpi^k_i)$ defines an important non-Hermitian quantities. More precisely,
set $C^k_{ij}:=G^{k\bar{l}}G_{i\bar{l}j}$. Then
\[
C:=C^k_{ij}\frac{\partial}{\partial\mathfrak{z}^k}\otimes d\mathfrak{z}^i\otimes d\mathfrak{z}^j
\]
is a section of $\pi^*T'M\otimes\pi^*T'^*M\otimes\pi^*T'^*M$, which is called the Cantor tensor. It is easy to check that $C=0$ if and only if $G$ is a Hermitian norm on $T'M$.

Theorem \ref{C-F} together with \cite[Corollary 3.10]{BC} then yields
\begin{theorem}\label{cur}
Let $\varpi$  and $\Omega$ denote the connection and curvature matrices of $\nabla$, respectively. Then we have

(1) $\partial \varpi=\varpi \wedge \varpi$.

(2) $\Omega=\bar{\partial}\varpi$ and $\bar{\partial}\Omega=0$.

(3) $\partial\Omega=-[\Omega,\varpi]$.

\end{theorem}

By a simple calculation, one has
\begin{align*}
\Omega^j_i
=:\,R^j_{ik\bar{l}}dz^k\wedge d\bar{z}^l+P^j_{ik,\bar{l}}dz^k\wedge \delta\bar{\xi}^l+S^j_{ik,\bar{l}}\delta\xi^k\wedge d\bar{z}^l+Q^j_{ik\bar{l}}\delta\xi^k\wedge \delta\bar{\xi}^l,\tag{2.3}\label{crfhe}
\end{align*}
where
\begin{align*}
&&R^j_{ik\bar{l}}=-\left(\frac{\delta \Gamma^j_{i,k}}{\delta \bar{z}^l}+{C^{j}_{is}}\frac{\delta N^s_k}{\delta \bar{z}^l}\right),&&S^j_{ik,\bar{l}}&=-\frac{\delta C^j_{ik}}{\delta \bar{z}^l},\\
&&P^j_{ik,\bar{l}}=-\left(\frac{\partial \Gamma^j_{i,k}}{\partial \bar{\xi}^l}+C^j_{is}\frac{\partial N^s_k}{\partial \bar{\xi}^l}\right),
&&Q^j_{ik\bar{l}}&=-\frac{\partial {C^j_{ik}}}{\partial\bar{\xi^l}}.
\end{align*}

\noindent \textbf{Notations}: Given $\alpha\in \mathscr{A}^\cdot(M)$, $\text{Re}(\alpha):=\frac12(\alpha+\bar{\alpha})$.
Given a smooth vector field $X$ on $M$, we use $\iota(X):\mathscr{A}^\cdot(M)\rightarrow \mathscr{A}^{\cdot-1}(M)$ to denote the {\it contraction operator} induced by $X$.

\section{The volume of an indicatrix}
Given a complex Finsler manifold $(M,F)$, for each $z\in M$, $(T'_zM,F_z)$ is a complex Minkowski space. The unit sphere in $(T'_zM,F_z)$ is the so-called {\it indicatrix} $S_zM$, i.e., $S_zM:=\{\xi\in T'_zM: F_z(\xi)=1\}$. This section is devoted to studying the indicatrices on $M$.

Let $(T'_zM,F_z)$ be as above.  Denote by $(\xi^i)$ the complex coordinate system of $T'_zM$, i.e., $\xi=(\xi^i)=\xi^i\frac{\partial}{\partial z^i}|_z$.
Then $F_z$ induces a Hermitian metric $h$ on $T'_zM-\{0\}$, which is defined by
\[
h=h_{i\bar{j}}\,d\xi^i\otimes d\bar{\xi}^j:=\frac{\partial^2F_z^2}{\partial \xi^i \partial \bar{\xi}^j}d\xi^i\otimes d\bar{\xi}^j=G_{i\bar{j}}(z,\xi)d\xi^i\otimes d\bar{\xi}^j.
\]
Clearly, $h$ is smooth at $\xi=0$ if and only if $F_z$ is a Hermitian norm. The volume form induced by $h$ is
\[
d\mu_{h}=\frac{1}{n!}\left(\frac{\sqrt{-1}}{2}h_{i\bar{j}}d\xi^i\wedge d\bar{\xi}^j \right)^n.
\]

To define the volume of $S_zM$, we now view $T'_zM$ as a real $2n$-dimensional vector space $V_\mathbb{R}$. More precisely, let \[
\mathcal {I}:V_{\mathbb{R}}\rightarrow T'_zM,\  (x^i,y^i)\mapsto(\xi^i)=(x^i+\sqrt{-1}y^i)
\]
 be the diffeomorphism. Then $h$ induces a Riemannian metric on $V_\mathbb{R}-\{0\}$, which is defined by
\[
\hat{g}_{v}(X,Y):=\frac{ h_{\mathcal {I}(v)}(\mathcal {I}_*(X),\mathcal {I}_*(Y))+\overline{h_{\mathcal {I}(v)}(\mathcal {I}_*(X),\mathcal {I}_*(Y)))}}{2},\   \forall X, Y\in T_vV_{\mathbb{R}}, \ v\neq 0.
\]

Let $S_{\mathbb{R}}:=\mathcal {I}^{-1}(S_zM)$. Note that $d\mu_{\hat{g}}=\mathcal {I}^*d\mu_h$, where $d\mu_{\hat{g}}$ is the Riemannian volume form induced by $\hat{g}$. Hence, it is natural to define the volume of $S_zM$ as the Riemannian volume of $S_{\mathbb{R}}$.  Explicitly, the volume of $S_zM$ is defined by
\[
\vol(z):=\int_{S_zM}d\nu_z,\
\]
where $d\nu_z:=\mathcal {I}^{-1*}\left[\iota(\textbf{n})\,d\mu_{\hat{g}}\right]$ and $\textbf{n}$ is the unit outward normal vector field along $S_{\mathbb{R}}$.

\begin{proposition}\label{volin}
For each $\xi\in S_zM$, we have
\[
d\nu_z|_{\xi}=\frac{\det h_{i\bar{j}}(\xi)}{2^{n-1}F_z^{2n}(\xi)}\cdot\text{Re}\left(({-1})^{\frac{-n^2}2}\sum_{i=1}^n(-1)^{i-1}\bar{\xi}^id\bar{\xi}^1\wedge \cdots \widehat{d\bar{\xi}^i}\cdots\wedge d\bar{\xi}^n \wedge d\xi^1\wedge \cdots \wedge d\xi^n\right).
\]
\end{proposition}
\begin{proof}Let $(u^i)$ denote the coordinates of $V_{\mathbb{R}}$, i.e., $u^{2i-1}:=x^i$ and $u^{2i}:=y^i$. Set $\hat{g}=\hat{g}_{ij}\,du^i\otimes du^j$. Thus,
\[
d\mu_{\hat{g}}:=\sqrt{\det \hat{g}}\,du^1\wedge\cdots\wedge du^{2n}.
\]
It is easy to see that ${\det \hat{g}}=(\det h)^2$.

Set $\mathcal {I}^{-1}(\xi)=(u^i)$.
Let $X(t)$, $t\in(-\epsilon,\epsilon)$ be a smooth curve on $S_\mathbb{R}$ with $X(0)=(u^i)$. By  Proposition \ref{cang}, we obtain
\begin{align*}
&2\hat{g}_{X(0)}(X(0),\dot{X}(0))\\
=&h_{\mathcal {I}(X(0))}(\mathcal {I}_*(\dot{X}(0)),\mathcal {I}_*(X(0)))+h_{\mathcal {I}(X(0))}(\mathcal {I}_*(X(0)),\mathcal {I}_*(\dot{X}(0)))\\
=&\frac{\partial F^2_{z}}{\partial \xi^i}\frac{d\mathcal {I}^i(X)}{dt}+\frac{\partial F^2_{z}}{\partial\bar{\xi}^i}\frac{\overline{d\mathcal {I}^i(X)}}{dt}=\left.\frac{d}{dt}\right|_{t=0}F^2_{z}(\mathcal {I}(X(t)))=0.
\end{align*}
Hence,
\[
\textbf{n}|_{(u^i)}=X(0)=u^i\frac{\partial}{\partial u^i}.
\]
Thus, we have
\begin{align*}
d\nu_z|_\xi&=\mathcal {I}^{-1*}[\iota(\textbf{n})\,d\mu_{\hat{g}}]\\
&=\det h_{i\bar{j}}(\xi)\cdot\mathcal {I}^{-1*}\left[\iota\left(u^i\frac{\partial}{\partial u^i}\right) du^1\wedge \cdots\wedge du^{2n}\right]\tag{3.1}\label{inreal}\\
&=\frac{\det h_{i\bar{j}}(\xi)}{F^{2n}_z(\xi)}\text{Re}\left(\frac{(-1)^{\frac{-n^2}2}}{2^{n-1}}\iota\left(\bar{\xi}^i\frac{\partial}{\partial \bar{\xi}^i}\right)d\bar{\xi}^1\wedge\cdots\wedge d\bar{\xi}^n\wedge d\xi^1\cdots d\xi^n\right).\tag{3.2}\label{in1}
\end{align*}
\end{proof}

The following example implies that the definition of the volume of a indicatrix is natural.
\begin{example}\label{firstex}Let $(T'_zM,F|_z)$ be as before.
Suppose that $F_z=\|\cdot\|$ is a Hermitian norm. Note that $d\nu_z$ is independent of the choice of coordinates of $T'_zM$.
Then we choose a orthonormal basis $\{e_i\}$ for $(T'_zM, \|\cdot\|)$. Let $(\xi^i)$ denote the coordinates of $T'_zM$ with  respect to $\{e_i\}$. Clearly,
$h(\frac{\partial}{\partial \xi^i},\frac{\partial}{\partial \xi^j})=\delta_{ij}$.
By Proposition \ref{volin}, we have
\begin{align*}
\vol(z)&=\int_{\|\xi\|=1}\text{Re}\left(\frac{({-1})^{\frac{-n^2}2}}{2^{n-1}}\sum_{i=1}^n(-1)^{i-1}\bar{\xi}^id\bar{\xi}^1\wedge \cdots \widehat{d\bar{\xi}^i}\cdots\wedge d\bar{\xi}^n \wedge d\xi^1\wedge \cdots \wedge d\xi^n\right)\\
&=\vol(\mathbb{S}^{2n-1}),
\end{align*}
where $\mathbb{S}^{2n-1}$ is the $(2n-1)$-dimensional unit Euclidean sphere. That is, the volumes of indicatrices on a Hermitian manifold are always equal to $\vol(\mathbb{S}^{2n-1})$.
\end{example}

The properties of the indicatrices of a complex Finsler manifold are much different from those of a real Finsler manifold.  For a real Finsler $n$-manifold $(N,\mathcal {F})$, if $\mathcal {F}$ is positively homogeneous of degree $1$, then the volume of indicatrix $\vol(x)$ can never exceed $\vol(\mathbb{S}^{n-1})$, with equality if and only if $\mathcal {F}|_x$ is a Euclidean norm (cf. \cite[Proposition 14.9.1]{BCS}). But this result cannot be generalized to the complex Finsler setting. For example, let $M=\mathbb{C}^2$ and $F(z,\xi):=\sqrt[4]{|\xi^1|^4+|\xi^2|^4}$, where $z=(z^1,z^2)$ and $\xi=\xi^1\frac{\partial}{\partial z^1}+\xi^2\frac{\partial}{\partial z^2}$. Then direct calculation yields $\vol(z)\equiv\vol(\mathbb{S}^3)$. But $F$ is not a Hermitian metric but  a complex locally Minkowski metric.

In general cases, $\vol(z)$ is nonconstant. However, the example above implies that for a complex locally Minkowski space, the volumes of the indicatrices may be a constant. Note that a locally Minkowski metric is a Berwald metric (see \cite{A2}). Then we have the following result.
\begin{proposition}
Let $(M,F)$ be a connected complex Berwald manifold. Then $\vol(z)$ is a constant.
\end{proposition}
\begin{proof}Given $p$ and $q$ in $M$, there exists a smooth curve $\gamma(t)$, $t\in[0,1]$ in $M$ such that $\gamma(0)=p$ and $\gamma(1)=q$. As in \cite{A2}, we define the covariant derivative of a smooth vector field $X=X^i\frac{\partial}{\partial z^i}$ along $\gamma$ by
\[
D_{\dot{\gamma}}X:=\left(\frac{dX^i}{dt}+X^j\Gamma^i_{j,k}(\gamma(t))\frac{dz^k\circ \gamma(t)}{dt}\right)\frac{\partial}{\partial z^i}.\tag{3.3}\label{prall}
\]
If $D_{\dot{\gamma}}X\equiv 0$, for all $t\in [0,1]$, then $X$ is called parallel along $\gamma(t)$. Let $P_t:T'_{p}M\rightarrow T'_{\gamma(t)}M$ denote the parallel translation induced by (\ref{prall}). It follows from \cite{A2} that $F(\gamma(t), P_t(X))=F(p,X)$ for all $X\in T'_pM$.

Denote by $(z^i,\xi^j)$ (resp. $(\varsigma^i,\zeta^j)$) a local coordinate system around $T'_pM$ (resp. $T'_{\gamma(t)}M$). Set $e_i:=\frac{\partial}{\partial z^i}|_p$ and $E_i(t):=P_t(e_i)$.
Since $\{E_i(t)\}$ is a basis for $T'_{\gamma(t)}M$, there exists non-singular matrices $(A^i_j(t))$ such that $E_i(t)=A_i^j \frac{\partial}{\partial \varsigma^j}|_{\gamma(t)}$. By (\ref{prall}), we have
\[
\frac{d A^i_s}{dt}+A^j_s\,\Gamma^i_{j,k}(\gamma(t))\frac{d\varsigma^k\circ \gamma(t)}{dt}=0\tag{3.4}\label{prafin}
\]

Recall $(\xi^i)$ and $(\zeta^i)$ are the coordinate systems of $T'_pM$ and $T'_{\gamma(t)}M$, respectively. From above, we have
\[
P_t(\xi^1,\cdots,\xi^n)=\left(\sum_i\xi^i\cdot A^1_i,\cdots,\sum_i\xi^i\cdot A^n_i\right)=(\zeta^1,\cdots,\zeta^n),
\]
which implies
\[
P_{t*}\frac{\partial}{\partial \xi^i}=A^k_i\frac{\partial}{\partial \zeta^k}, \ P_t^*d\zeta^i=A^i_j\,d\xi^j.\tag{3.5}\label{parll2}
\]

 By Theorem \ref{C-F}, (\ref{parll2}) and (\ref{prafin}), one can check that
\[
\frac{d}{dt}\left[h_{P_t(\xi)}\left(P_{t*}\frac{\partial}{\partial \xi^i},P_{t*}\frac{\partial}{\partial \xi^j}\right)\right]=0.\tag{3.6}\label{prall4}
\]

Set $B_{\gamma(t)}(r):=\{\zeta\in T'_{\gamma(t)}M: F(\gamma(t),\zeta)<r\}$. Note that $\det h_{i\bar{j}}({\gamma(t)},\zeta)$ can be view as a function on $S_{\gamma(t)}M$. Since $S_{\gamma(t)}M$ is compact, $|\det h_{i\bar{j}}({\gamma(t)},\zeta)|$ is bounded. Hence,
\[
\lim_{\epsilon\rightarrow 0^+}\int_{B_{\gamma(t)}(\epsilon)} d\mu_{\hat{g}}({\gamma(t)})=0.\tag{3.7}\label{parll3}
\]
Stokes' formula together with (\ref{parll3}), (\ref{prall4}) and (\ref{parll2}) then yields
\begin{align*}
&\vol(\gamma(t))=2\,\text{Re}\left(\int_{S_{\gamma(t)}M}\iota\left(\bar{\zeta}^i\frac{\partial}{\partial \bar{\zeta}^i}\right)d\mu_{h}({\gamma(t)})\right)\\
=&2n\,\text{Re}\left(\int_{B_{\gamma(t)}(1)}d\mu_{h}({\gamma(t)})\right)=2n\,\text{Re}\left(\int_{P_t(B_p(1))}d\mu_{h}({\gamma(t)})\right)\\
=&2n\,\text{Re}\left(\frac{(-1)^\frac{-n^2}{2}}{2^n}\int_{B_p(1)}P_t^*\left(\det h_{P_t(\xi)}\left(\frac{\partial}{\partial \zeta^i},\frac{\partial}{\partial \zeta^j}\right)d\bar{\zeta}^1\wedge\cdots \wedge d\bar{\zeta}^n\wedge d\zeta^1\wedge\cdots \wedge d{\zeta}^n \right)\right)\\
=&2n\,\text{Re}\left(\int_{B_p(1)} d\mu_{h}(p)\right)=\vol(p).
\end{align*}
We are done by letting $t=1$.\end{proof}

Now we define a form $\sigma_z$ on $T'_zM-\{0\}$ by
\[
\sigma_z:=r^*d\nu_z,\tag{3.8}\label{dekey}
\]
where $r:T'_zM-\{0\}\rightarrow S_zM$, $\xi\mapsto \xi/F_z(\xi)$.
\begin{proposition}\label{key}
For each $\xi\in T'_zM-\{0\}$, we have
\[
\sigma_z|_{\xi}=\frac{\det h_{i\bar{j}}(\xi)}{2^{n-1}F_z^{2n}(\xi)}\cdot\text{Re}\left(({-1})^{\frac{-n^2}2}\sum_{i=1}^n(-1)^{i-1}\bar{\xi}^id\bar{\xi}^1\wedge \cdots \widehat{d\bar{\xi}^i}\cdots \wedge d\bar{\xi}^n \wedge d\xi^1\wedge \cdots \wedge d\xi^n\right).
\]
\end{proposition}
\begin{proof}Denote by $(u^i)$ the coordinate system of $V_{\mathbb{R}}$, i.e., $u^{2i-1}=x^i$ and $u^{2i}=y^i$. For $\xi\neq 0$, set $\mathcal {I}(u):=\xi$, $L(\xi):=F^{-1}_{z}(\xi)$ and $L(u):=L(\mathcal {I}(u))$.
Since $r(\xi)=L(\xi)\cdot \xi$, we have
\begin{align*}
\mathcal {I}^*(\sigma_z|_\xi)
=\frac{\det h}{F^{2n}_{z}(\mathcal {I}(u))}\frac{\sum_{i=1}^{2n}(-1)^{i-1}[L(u)u^i]d[L(u)u^1]\wedge\cdots \widehat{d[L(u)u^i]}\cdots\wedge  d[L(u)u^{2n}]}{L^{2n}(u)}.
\end{align*}
Set $\beta:=(\det h/F^{2n}_{z}(\mathcal {I}(u)))^{-1}\cdot\left[\mathcal {I}^*(\sigma_z)-\alpha\right]$, where
\[
\alpha:=\frac{\det h}{F^{2n}_{z}(\mathcal {I}(u))}{\sum_{i=1}^{2n}(-1)^{i-1}u^idu^1\wedge\cdots \widehat{du^i}\cdots\wedge  du^{2n}}.
\]

Since $dL\wedge dL=0$, we have $\beta=\sum_{i\neq j}I_{ij}$, where
\begin{align*}
I_{ij}=&\frac{(-1)^{i-1}[L(u)u^i][du^1\cdot L(u)]\wedge\cdots\wedge [dL(u)\cdot u^j]\wedge \cdots\wedge\widehat{[d u^i\cdot L(u)]}\wedge\cdots\wedge [d u^n \cdot L(u)]}{L^n(u)}\\
=&\left\{
\begin{array}{lll}
&\frac{(-1)^{i-1}u^iu^j du^1\wedge \cdots\wedge du^{j-1}\wedge dL\wedge du^{j+1}\wedge\cdots\wedge \widehat{du^{i}}\wedge\cdots\wedge du^n}{L(u)},&i>j,\\
&\frac{(-1)^{i-1}u^iu^j du^1\wedge \cdots\wedge \widehat{du^{i}}\wedge\cdots\wedge du^{j-1}\wedge dL\wedge du^{j+1}\wedge\cdots\wedge du^n}{L(u)},&j<i.
\end{array}
\right.
\end{align*}
Since $I_{ij}+I_{ji}=0$, $\beta=0$. Hence, $\sigma_z=\mathcal {I}^{-1*}\alpha$. The conclusion then follows from (\ref{inreal}) and (\ref{in1}).
\end{proof}

$\sigma_z$ plays a pivotal role in calculating the local degree of a holomorphic vector field at a non-degenerate
zero, which allows us to establish the Gauss-Bonnet-Chern theorem. See Section 5 for more details.

\section{The exterior differential equations}
Let $(M,F)$ be a closed complex Finsler $n$-manifold and
let $X$ be a holomorphic vector field on $M$ with non-degenerate zeros $\{\zeta_1,\cdots,\zeta_k\}$.
Set $M_0:=M-\cup_{i=1}^k\{\zeta_i\}$.

We first recall the definition of tensorial matrices (see \cite[Definition 4.1]{CCL}).
\begin{definition}
If for every local frame field $S$ of a vector bundle $E\overset{\mathfrak{p}}{\rightarrow} N$, there is a given $n\times n$ matrix $\Phi_S$ of exterior differential $k$-forms satisfying
\[
\left.
\begin{array}{lll}
&S'=A\cdot S\\
&\det A\neq 0
\end{array}
\right\}
\Rightarrow \Phi_{S'}=A\cdot \Phi_{S} \cdot A^{-1},
\]
then we call $\{\Phi_S\}$ is a tensorial matrix on $N$ with respect to $E$.
\end{definition}

We now construct a tensorial matrix on $M_0$.
Choose an arbitrary local coordinate system $(z^i)$ of $M$, set
\[
\Theta:=-\iota(X)X^*\varpi-\mu,
\]
where $\mu:=(\frac{\partial X^i}{\partial z^j})$ and $\varpi=(\varpi^i_j)$ is the Chern-Finsler connection $1$-form matrix with respect to $\{\partial/\partial{z}^i\}$.
\begin{proposition}\label{impr}
$\Theta$ determines a tensorial matrix defined on $M_0$ with respect to the bundle $T'M_0$.
\end{proposition}
\begin{proof}Let $(\varsigma^i)$ be another coordinate on $M$ and set $A=A(z^i):=\left(\frac{\partial z^j}{\partial \varsigma^i}\right)$. It is easy to check that
\begin{align*}
dA^j_i=\frac{\partial^2 z^j}{\partial \varsigma^i\partial \varsigma^k} \cdot(A^{-1})^k_l\cdot dz^l.\tag{4.1}\label{trans1}
\end{align*}

Denote by $\widetilde{\varpi}$ the Chern-Finsler connection $1$-forms matrix with respect to $\{{\partial}/{\partial {\varsigma}^i}\}$. Clearly,
\[
\widetilde{\varpi}=(d A+A\cdot\varpi)\cdot A^{-1}.
\]
Since $A=A(z^i)$ and $X^*dz^i=dz^i$, we have
\[
X^*\widetilde{\varpi}=(d A+A\cdot X^*\varpi)\cdot A^{-1},
\]
which implies
\[
\iota(X) X^*\widetilde{\varpi}=(\iota(X)d A+A\cdot \iota(X)X^*\varpi)\cdot A^{-1}.\tag{4.2}\label{trans2}
\]
Set
\[
X=:\widetilde{X}^i\frac{\partial}{\partial \varsigma^i},\ \widetilde{\mu}^j_i:=\frac{\partial \widetilde{X}^j}{\partial \varsigma^i}.
\]
Note that
\begin{align*}
\widetilde{X}^i=X^k \cdot(A^{-1})_k^i,\
\frac{\partial}{\partial \varsigma^i}=A_i^j\frac{\partial}{\partial z^j}. \tag{4.3}\label{trans3}
\end{align*}
(\ref{trans3}) together with (\ref{trans1}) then yields
\begin{align*}
\widetilde{\mu}^i_j&=A^m_j\cdot\frac{\partial X^k}{\partial z^m}\cdot(A^{-1})^i_k-X^k\cdot (A^{-1})^s_k\cdot A^m_j\cdot \frac{\partial A^l_s}{\partial z^m}\cdot (A^{-1})^i_l\\
&=A^m_j\cdot\mu^k_m\cdot(A^{-1})^i_k-X^k\cdot (A^{-1})^s_k\cdot \frac{\partial^2z^l}{\partial\varsigma^s\partial\varsigma^j}\cdot  (A^{-1})^i_l\\
&=A^m_j\cdot\mu^k_m\cdot(A^{-1})^i_k-\iota(X)dA_j^l\cdot (A^{-1})^i_l,
\end{align*}
that is,
\[
\widetilde{\mu}=(A\cdot\mu -\iota(X)dA)\cdot A^{-1}.\tag{4.4}\label{trans4}
\]
Combining (\ref{trans2}) with (\ref{trans4}), we obtain
\[
-\widetilde{\mu}- \iota(X) X^*\widetilde{\varpi}=A\cdot(-\mu- \iota(X)X^*\varpi)\cdot A^{-1},
\]
Hence, $\Theta^j_i\frac{\partial}{\partial z^j}\otimes dz^i$ is a tensor filed on $M_0$, which determines a tensorial matrix.
\end{proof}

\begin{lemma}Let $\Omega$ denote the curvature matrix of the Chern-Finsler connection. Thus, on $M_0$, we have
\[
\bar{\partial}\Theta=\iota(X)[X^*\Omega].\tag{4.5}\label{usef2}
\]
\end{lemma}
\begin{proof}Recall that $\iota(X)\circ \bar{\partial}+\bar{\partial}\circ \iota(X)=0$ (see \cite{B}). By induction hypothesis, one can show that
\[
X^*\circ \bar{\partial}-\bar{\partial}\circ X^*=0,\tag{4.6}\label{usef}
\]
where the first $\bar{\partial}$ is defined on $\mathscr{A}(T'M_0)$ while the second $\bar{\partial}$ is defined on $\mathscr{A}(M_0)$. From above, we obtain
\[
\bar{\partial}\Theta=-\bar{\partial}[\iota(X)X^*\varpi]=\iota(X)[\bar{\partial}X^*\varpi]=\iota(X)[X^*\bar{\partial}\varpi]=\iota(X)[X^*\Omega].
\]
\end{proof}

As in \cite{BC}, for any $ A,B\in \mathcal {M}_{n\times n}$ and $\lambda\in \mathbb{C}$, we define
\[
\det(\lambda A+B)=:\sum_{j=0}^{n}\lambda^j {\det}^j(A;B), \ A,B\in \mathcal {M}_{n\times n}.
\]

Since $\Theta$ and $X^*\Omega$ are two tensorial matrices on $M_0$, all $\det^j(X^*\Omega;\Theta)$s are well-defined. Inspired by the success of the constructions in \cite{C}, we define the following polynomials
\[
\Psi_j:=\omega_X\wedge (\bar{\partial}\omega_X)^{n-j-1}\wedge {\det}^j\left(\frac{\sqrt{-1}}{2\pi}X^*\Omega;\frac{\sqrt{-1}}{2\pi}\Theta\right)\in \mathscr{A}^{n,n-1}(M_0), \ 0\leq j\leq n-1,
\]
where
\[
\omega_X:=\frac{g_X(\cdot,X)}{F^2(X)}\in \mathscr{A}^{1,0}(M_0).
\]

\begin{lemma}\label{Tanss}
Let $c_n(\pi^*T'M;\nabla)$ denote the $n$-th Chern form of $\pi^*T'M$ defined by the Chern-Finsler connection. Thus, on $M_0$, we have
\[
X^*c_n(\pi^*T'M;\nabla)=-d\Psi,
\]
where $\Psi:=\sum_{j=0}^{n-1}\Psi_j$.
\end{lemma}
\begin{proof}Note that $\bar{\partial}\Omega=0$ and $\iota(X)\Theta=0$. Thus,
(\ref{usef}) together with (\ref{usef2}) yields
\[
\bar{\partial}{\det}^j(X^*\Omega;\Theta)=\iota(X) {\det}^{j+1}(X^*\Omega;\Theta).
\]
Since $\Psi\in \mathscr{A}^{n,n-1}(M_0)$, we have
\begin{align*}
d\Psi=\bar{\partial}\Psi=&-\sum_{j=1}^{n}\omega_X\wedge(\bar{\partial}\omega_X)^{n-j}\wedge \iota(X){\det}^{j}\left(\frac{\sqrt{-1}}{2\pi}X^*\Omega;\frac{\sqrt{-1}}{2\pi}\Theta\right)\\
&+\sum_{j=0}^{n-1}(\bar{\partial}\omega_X)^{n-j}\wedge {\det}^j\left(\frac{\sqrt{-1}}{2\pi}X^*\Omega;\frac{\sqrt{-1}}{2\pi}\Theta\right).
\end{align*}
Note that $\iota(X)\circ\bar{\partial}\omega_X=-\bar{\partial}\circ\iota(X)\omega_X=-\bar{\partial}1=0$. Hence, we obtain
\begin{align*}
\iota(X)\circ d\Psi=&-\sum_{j=1}^{n}(\bar{\partial}\omega_X)^{n-j}\wedge \iota(X){\det}^{j}\left(\frac{\sqrt{-1}}{2\pi}X^*\Omega;\frac{\sqrt{-1}}{2\pi}\Theta\right)\\
&+\sum_{j=0}^{n-1}(\bar{\partial}\omega_X)^{n-j}\wedge\iota(X) {\det}^j\left(\frac{\sqrt{-1}}{2\pi}X^*\Omega;\frac{\sqrt{-1}}{2\pi}\Theta\right)\\
=&-\iota(X){\det}\left(\frac{\sqrt{-1}}{2\pi}X^*\Omega\right).
\end{align*}
That is
\[
\iota(X)\left[d\Psi+X^*c_n(\pi^*T'M;\nabla)\right]=0.
\]
Since $[d\Psi+X^*c_n(\pi^*T'M;\nabla)]$ is a form of top degree and $X$ dose not vanishes on $M_0$, the conclusion follows.
\end{proof}

\begin{remark}By the Bott-Chern form \cite{BC}, one can show
that there exists a form $\eta\in \mathscr{A}(T'M\backslash0)$ such that $c_n(\pi^*T'M;\nabla)=d\eta$. Thus, $X^*c_n(\pi^*T'M;\nabla)=dX^*\eta$, which is similar to Lemma \ref{Tanss}. But it seems impossible to express $\eta$ in term of $dz^i,d \xi^j$, which makes it indeed difficult to calculate $\int_MdX^*\eta$ in the Finsler setting. However,
for a Hermitian manifold, without knowing the explicit formula of $\eta$, one can still calculate the value of $\int_MdX^*\eta$ by constructing a trivial bundle. See \cite{BC,Co} for more details.
\end{remark}

Since $c_n(\pi^*T'M;\nabla)$ is real, we have
\[
X^*c_n(\pi^*T'M;\nabla)=\overline{X^*c_n(\pi^*T'M;\nabla)}=-\overline{d\Psi}=-d\bar{\Psi},
\]
which implies
\[
X^*c_n(\pi^*T'M;\nabla)+\sum_{j=1}^{n-1}d\,\text{Re}(\Psi_j)=-d\,\text{Re}(\Psi_0).\tag{4.7}\label{trans4.7}
\]

Let $\varpi_H$ and $\varpi_V$ denote the horizontal component and the vertical component of $\varpi$, respectively.  Recall that $\varpi_H$ is the Rund connection $1$-forms matrix. Hence,
the same argument as in Proposition \ref{impr} shows
\[
\Theta_H:=-\iota(X)X^*\varpi_H-\mu,\ \Theta_V:=\Theta-\Theta_H=-\iota(X)X^*\varpi_V
\]
are two tensorial matrices.
Hence, all $\det^j(\Theta_V;\Theta_H)$s, $0\leq j\leq n$ are well-defined and therefore,
\begin{align*}
\Psi_0&=\left(\frac{\sqrt{-1}}{2\pi}\right)^n\det(\Theta)\,\omega_X\wedge (\bar{\partial}\omega_X)^{n-1}\\
&=\left(\frac{\sqrt{-1}}{2\pi}\right)^n\sum_{j=0}^{n-1}{\det}^j (\Theta_V;\Theta_H)\,\omega_X\wedge (\bar{\partial}\omega_X)^{n-1}+\left(\frac{\sqrt{-1}}{2\pi}\right)^n\det(\Theta_H)\,\omega_X\wedge (\bar{\partial}\omega_X)^{n-1}\\
&=:\Lambda_1+\Lambda_2.
\end{align*}
In view of (\ref{trans4.7}) and the definition of currents (cf. \cite{D}), we have the following result.
\begin{proposition}\label{Eform}
Define a current $\mathfrak{E}$ of order $0$ on $M$ by
\[
\mathfrak{E}:={\sum_{j=1}^{n-1}d\text{Re}(\Psi_j)+d\text{Re}(\Lambda_1)+d\log \vol(z)\wedge \text{Re}(\Lambda_2)}.\tag{4.8}\label{correctitem}
\]
Then we have the following equality on $M$ (in the sense of current)
\[
\left[\frac{X^*c_n(\pi^*T'M)+\mathfrak{E}}{\vol(z)}\right]=-d\left(\frac{\text{Re}(\Lambda_2)}{\vol(z)}\right).\tag{4.9}\label{transsgr}
\]
Here, $\vol(z)$ is the volume of the indicatrix $S_zM$ defined in Section 3.
\end{proposition}

\section{The Gauss-Bonnet-Chern theorem}
In this section, we shall establish a Gauss-Bonnet-Chern theorem for complex Finsler manifolds. Let $(M,F)$ and $X$ be as in Section 4. First, we have

\begin{lemma}\label{smallemma}
For each zero $\zeta$ of $X$, we have
\[
\lim_{z\rightarrow \zeta}\frac{\det(\Theta_H)}{\det (\frac{\partial X}{\partial z})}=(-1)^n.
\]
\end{lemma}
\begin{proof}Let $(z^i,\xi^j)$ be a local coordinate system of $T'\overline{U}$, where $U$ is a small neighborhood of $\zeta$ such that $\overline{{U}}$ is compact. It follows from (\ref{Nformula}) that $N^j_k$ is homogeneous of degree $1$ in $\xi$ and homogeneous of degree $0$ in $\bar{\xi}$. Since $\Gamma^j_{i,k}={\partial N^j_k}/{\partial \xi^i}$, $\Gamma^j_{i,k}$ is homogeneous of degree $0$ in both $\xi$ and $\bar{\xi}$. Thus, $\Gamma^j_{i,k}$ can be viewed as a function on $P\overline{{U}}:=T'\overline{U}/(\mathbb{C}-\{0\})$. Since $P\overline{{U}}$ is compact, $|\Gamma^j_{i,k}|$ can attend its maximum $L<\infty$ and therefore,
\[
0\leq \lim_{z\rightarrow \zeta}|\Gamma^i_{j,k}X^k|\leq \lim_{z\rightarrow \zeta}L\cdot\sum_k|X^k|=0.
\]
Since
\[
(\Theta_H)^i_j=-\frac{\partial X^i}{\partial z^j}-\Gamma^i_{j,k}(z,X)X^k,
\]
the conclusion then follows.
\end{proof}

The following is a key lemma to the proof of the Gauss-Bonnet-Chern theorem.
\begin{lemma}\label{degreein}
For each zero $\zeta$ of $X$, we have
\[
\lim_{\epsilon\rightarrow 0^+}\int_{\partial B_\epsilon(\zeta)}\frac{\text{Re}(\Lambda_2)}{\vol(z)}=\frac{1}{\vol(\mathbb{S}^{2n-1})}.
\]
where $B_\epsilon(\zeta)$ is the ball of radius $\epsilon$ centered at $\zeta$ defined by a local coordinate neighborhood of $\zeta$.
\end{lemma}
\begin{proof}
Fix a small $\delta>0$. Let $(z^i)$ be a coordinate system of $\overline{B_\delta(\zeta)}$ and let $(z^i,\xi^j)$ denote the corresponding coordinate system of $T'\overline{B_\delta(\zeta)}$.

For simplicity, set $\bar{X}_i:=g_{i\bar{j}}(X)\bar{X}^j$. Thus,
\[
\omega_X=\frac{\bar{X}_idz^i}{\bar{X}_iX^i},\ \bar{\partial}\omega_X=\frac{\bar{\partial}\bar{X}_j\wedge dz^j}{\bar{X}_iX^i}-\frac{X^i\bar{\partial}\bar{X}_i\wedge \bar{X}_jdz^j}{(\bar{X}_iX^i)^2}.
\]
Since $( \bar{X}_jdz^j)\wedge ( \bar{X}_idz^i)=0$, we obtain
\begin{align*}
&\omega_X\wedge (\bar{\partial}\omega_X)^{n-1}   \\
=&(-1)^{\frac{n(n-1)}2}(n-1)!\frac{\sum_i(-1)^{i-1}\bar{X}_i\bar{\partial}\bar{X}_1\wedge \cdots\widehat{\bar{\partial}\bar{X}_i}\cdots \wedge \bar{\partial}\bar{X}_n \wedge dz^1\wedge\cdots\wedge dz^n }{F^{2n}(X)}\\
=&(-1)^{\frac{n(n-1)}2}\frac{(n-1)!}{\det (\frac{\partial X}{\partial z})}\frac{\sum_i(-1)^{i-1}\bar{X}_i\bar{\partial}\bar{X}_1\wedge \cdots\widehat{\bar{\partial}\bar{X}_i}\cdots \wedge \bar{\partial}\bar{X}_n \wedge dX^1\wedge\cdots\wedge dX^n}{F^{2n}(X)}.
\end{align*}

For $\xi\neq0$, set $\bar{\xi}^i=g_{i\bar{j}}(z,\xi)\bar{\xi}^j$. Note that $X:\partial B_\epsilon(\zeta)\rightarrow T'M\backslash0$ is a embedding map and $X^*(\bar{\partial}\bar{\xi}_i)=\bar{\partial}(X^*\bar{\xi}_i)=\bar{\partial}\bar{X}_i$. Lemma \ref{smallemma} then yields
\[
\lim_{\epsilon\rightarrow0^+}\int_{\partial B_\epsilon(\zeta)}\frac{\Lambda_2}{\vol(z)}=(-1)^{\frac{n^2}2+n}\frac{(n-1)!}{(2\pi)^n}\lim_{\epsilon\rightarrow0^+}\int_{X(\partial B_\epsilon(\zeta))}\eta,\tag{5.1}\label{new5.1}
\]
where
\[
\eta:=\frac{\sum_i(-1)^{i-1}\bar{\xi}_i\bar{\partial}\bar{\xi}_1\wedge \cdots\widehat{\bar{\partial}\bar{\xi}_i}\cdots \wedge \bar{\partial}\bar{\xi}_n \wedge d\xi^1\wedge\cdots\wedge d\xi^n}{F^{2n}_z(\xi)\cdot \vol(z)}.
\]
Note that
\[
\bar{\partial}\bar{\xi}_i=\bar{\xi}^j\frac{\partial g_{i\bar{j}}}{\partial \bar{z}^k}d\bar{z}^k+ g_{i\bar{j}} d\bar{\xi}^j.
\]
Since $\epsilon\rightarrow 0^+$, $d\bar{z}$-parts of $\eta$ do not contribute. Hence,
\[
\lim_{\epsilon\rightarrow 0^+}\int_{X(\partial B_\epsilon(\zeta))}\eta=\lim_{\epsilon\rightarrow 0^+}\int_{X(\partial B_\epsilon(\zeta))}\text{ pure $(d\xi,d\bar{\xi})$-parts of }\eta.
\]
A direct calculation yields that
\begin{align*}
&\text{ pure $(d\xi,d\bar{\xi})$-parts of }\eta\\
=&\frac{(\det g_{i\bar{j}}(z,\xi))\sum_i(-1)^{i-1}\bar{\xi}^id\bar{\xi}^1\wedge \cdots\wedge\widehat{d\bar{\xi}^i}\wedge\cdots \wedge d\bar{\xi}^n \wedge d\xi^1\wedge\cdots\wedge d\xi^n}{F_z^{2n}(\xi)\cdot \vol(z)}\\
=&f(z,\xi) \cdot\tilde{\eta},
\end{align*}
where
\begin{align*}
f(z,\xi)&:=\left(\frac{F^{2n}_\zeta(\tilde{\xi})\cdot\vol(\zeta)\cdot(\det g_{i\bar{j}})(z,{\xi})}{F^{2n}_z(\xi)\cdot\vol(z)\cdot(\det g_{i\bar{j}})(\zeta,\tilde{\xi})}\right),\\
\tilde{\eta}&:=\frac{(\det g_{i\bar{j}})(\zeta,\tilde{\xi})\sum_i(-1)^{i-1}\bar{\xi}^id\bar{\xi}^1\wedge \cdots\widehat{d\bar{\xi}^i}\cdots \wedge d\bar{\xi}^n \wedge d\xi^1\wedge\cdots\wedge d\xi^n}{F_\zeta^{2n}(\tilde{\xi})\cdot\vol(\zeta)},
\end{align*}
and $\tilde{\xi}$ is a vector in $T'_{\zeta}M$ defined by $\xi$, i.e., $\tilde{\xi}:=\xi^i \frac{\partial}{\partial z^i}|_{\zeta}$.
Clearly, $\lim_{z\rightarrow \zeta}f(z,\xi)=1$. Hence,
\[
\lim_{\epsilon\rightarrow 0^+}\int_{X(\partial B_\epsilon(\zeta))}\eta=\lim_{\epsilon\rightarrow 0^+}\int_{X(\partial B_\epsilon(\zeta))}\tilde{\eta}.\tag{5.2}\label{new5.2}
\]

Define a local trivialization $\varphi:T'M|_{B_\delta(\zeta)}\rightarrow B_\delta(\zeta)\times T'_\zeta M$ by
\[
\varphi(z,\xi)=\varphi\left(\left.\xi^j\frac{\partial}{\partial z^i}\right|_z\right)=\left(z,\left.\xi^i\frac{\partial}{\partial z^i}\right|_\zeta\right).
\]
Then we obtain a map $\phi:T'M|_{B_\delta(\zeta)}\rightarrow  T'_\zeta M$, $\phi\left(\left.\xi^j\frac{\partial}{\partial z^i}\right|_z\right)=\left.\xi^i\frac{\partial}{\partial z^i}\right|_\zeta$. Let $\sigma_\zeta$ be the form on $T'_\zeta M-\{0\}$ defined as in (\ref{dekey}). Proposition \ref{key} then yields
\[
\sigma_\zeta=\frac{(\det g_{i\bar{j}})(\zeta,{\xi})}{2^{n-1}F_\zeta^{2n}({\xi})}\text{Re}\left(({-1})^{\frac{-n^2}2}{\sum_i(-1)^{i-1}\bar{\xi}^id\bar{\xi}^1\wedge \cdots\widehat{d\bar{\xi}^i}\cdots \wedge d\bar{\xi}^n \wedge d\xi^1\wedge\cdots\wedge d\xi^n}\right).
\]
Thus, we have
\[
\text{Re}\left((-1)^{\frac{-n^2}{2}}\tilde{\eta}\right)=2^{n-1}\phi^*\left(\frac{\sigma_\zeta}{\vol(\zeta)}\right).\tag{5.3}\label{new5.3}
\]
(\ref{new5.1}) together with (\ref{new5.2}) and (\ref{new5.3}) then yields that
\begin{align*}
&\lim_{\epsilon\rightarrow 0^+}\int_{\partial B_\epsilon(\zeta)}\frac{\text{Re}(\Lambda_2)}{\vol(z)}=\lim_{\epsilon\rightarrow 0^+}\text{Re}\left(\int_{\partial B_\epsilon(\zeta)}\frac{\Lambda_2}{\vol(z)}\right)\\
=&\frac{1}{\vol(\mathbb{S}^{2n-1})\vol(\zeta)}\lim_{\epsilon\rightarrow 0^+}\int_{X(\partial B_\epsilon(\zeta))}\phi^*{\sigma_\zeta}\\
=&\frac{1}{\vol(\mathbb{S}^{2n-1})\vol(\zeta)}\lim_{\epsilon\rightarrow 0^+}\int_{\partial B_\epsilon(\zeta)}(\phi\circ X)^*\sigma_\zeta.\tag{5.4}\label{new5.4}
\end{align*}

Let $r:T'_\zeta M-\{0\}\rightarrow S_\zeta M$, $\xi\mapsto \xi/F_\zeta(\xi)$.
It is noticeable that $\text{deg}(r\circ \phi\circ X)$ is exactly the local degree of $X$ at $\zeta$, which always
equals to $+1$. (cf. \cite{Ba,BC}).
Hence, (\ref{dekey}) yields
\begin{align*}
&\int_{\partial B_\epsilon(\zeta)}(\phi\circ X)^*\sigma_\zeta=\int_{\partial B_\epsilon(\zeta)}(\phi\circ X)^*\circ r^*d\nu_\zeta\\
=&\text{deg}(r\circ \phi\circ X)\int_{S_\zeta M}d\nu_\zeta=\vol(\zeta).\tag{5.5}\label{new5.5}
\end{align*}
The conclusion then follows from (\ref{new5.4}) and (\ref{new5.5}).
\end{proof}
\begin{remark}If $F$ is Hermitian, then $\vol(z)=\vol(\mathbb{S}^{2n-1})$ and therefore, one can shows Lemma \ref{degreein} immediately by using the Bochner-Martinelli formula (cf. \cite{C}). But the Bochner-Martinelli formula seems invalid in the Finsler setting, since $\vol(z)$ may not be a holomorphic function and $F_z$ is much different from a Hermitian norm.
\end{remark}

It is remarkable that the Hopf theorem implies (cf. \cite{Ba})
\[
\chi(M)=\text{Number of zeroes of }X.
\]
Thus, Proposition \ref{Eform} together with Lemma \ref{degreein} furnishes the following theorem.
\begin{theorem}\label{Gsb}Let $(M,F)$ be a closed complex Finsler $n$-manifold and let $X$ be a holomorphic vector field on $M$ with non-degenerate zeros. Then we have
\[
\int_M\left[\frac{X^*c_n(\pi^*T'M;\nabla)+\mathfrak{E}}{\vol(z)}\right]=\frac{\chi(M)}{\vol(\mathbb{S}^{2n-1})},
\]
where $\mathfrak{E}$ is a current on $M$ dependent only on $X$ and $F$.
\end{theorem}
\begin{proof}
Set $M_\epsilon:=M-\cup_{i}B_\epsilon(\zeta_i)$.
It follows from Proposition \ref{Eform}, Lemma \ref{degreein} and the Hopf theorem that
\begin{align*}
&\lim_{\epsilon\rightarrow 0^+}\int_{M_\epsilon}\left[\frac{X^*c_n(\pi^*T'M;\nabla)+\mathfrak{E}}{\vol(z)}\right]=-\lim_{\epsilon\rightarrow 0^+}\int_{M_\epsilon}d\left(\frac{\text{Re}(\Lambda_2)}{\vol(z)}\right)\\
=&\sum_i\lim_{\epsilon\rightarrow 0^+}\int_{\partial B_\epsilon(\zeta_i)}\frac{\text{Re}(\Lambda_2)}{\vol(z)}=\frac{\text{Number of zeroes of } X}{\vol(\mathbb{S}^{2n-1})}=\frac{\chi(M)}{\vol(\mathbb{S}^{2n-1})}.
\end{align*}
\end{proof}

\begin{example}\label{realex}Consider the product manifold $M:=\mathbb{T}\times \mathbb{T}$, where $\mathbb{T}=\mathbb{C}/\mathcal {L}$ is a complex torus of dimension $1$.
Define a complex Finsler on $M$ by
\[
F(z,\xi):=\sqrt[4]{|\xi^1|^4+|\xi^2|^4}
\]
where $\xi=\sum_{i=1}^2\xi^i\frac{\partial}{\partial z^i}\in T'M$ and $z=(z^1,z^2)$ is the natural coordinate system of $M$. Since $F$ is not Hermitian but local Minkowski,  the curvature of $\nabla$ is equal to
\[
\Omega^j_i=-\frac{\partial C^j_{ik}}{\partial\bar{\xi}^l}d\xi^k\wedge d\bar{\xi}^l\neq 0.
\]
In particular, $c_2(\pi^*T'M;\nabla)=\frac{-1}{8\pi^2}(\Omega^i_i\wedge \Omega^j_j-\Omega_i^j\wedge\Omega^i_j )$ does not vanish.

Given $(\lambda^1,\lambda^2)\in \mathbb{C}\times \mathbb{C}-\{(0,0)\}$, set $X:=\sum_{i=1}^2\lambda^i\frac{\partial}{\partial z^i}$. Thus, $X$ is a holomorphic vector filed on $M$ (cf.\,\cite[Example 4.4.4]{V}). Since $X^i\equiv\lambda^i$ and $\vol(z)\equiv4\pi$, it is not hard to check that $X^*c_2(\pi^*T'M;\nabla)=0$ and $\mathfrak{E}=0$. Note that $\chi(M)=0$. Hence, we have
\[
\int_M\left[\frac{X^*c_2(\pi^*T'M;\nabla)+\mathfrak{E}}{\vol(z)}\right]=\frac{\chi(M)}{\vol(\mathbb{S}^{3})}.
\]
\end{example}

In the Hermitian case, Theorem \ref{Gsb} reduces to (\ref{1.3}).
\begin{theorem}[\cite{BC}] Let $(M,H)$ be a closed complex $n$-dimensional Hermitian manifold. Then
\[
\int_M c_n(T'M)=\chi(M),
\]
where $c_n(T'M)$ denotes the $n$th-Chern form defined by the canonical Hermitian connection.
\end{theorem}
\begin{proof}Clearly, $g=\pi^*H$ is the Hermitian metric on $\pi^*T'M$ induced by $F:=\sqrt{H}$. Let $D$ denote the canonical Hermitian connection on $T'M$. Thus, the Chern-Finsler connection $\nabla=\pi^*D$. Hence,
\[
X^*c_n(\pi^*T'M;\nabla)=X^*\pi^*c_n(T'M)=c_n(T'M).
\]
We now claim that
\[
\int_M {\mathfrak{E}}=0.\tag{5.6}\label{newin}
\]
Recall that Example \ref{firstex} implies that $\vol(z)\equiv \vol(\mathbb{S}^{2n-1})$. Hence, if (\ref{newin}) is true, then we are done by Theorem \ref{Gsb}.

 Since $H$ is a Hermitian metric, the Cantor tensor $C$ vanishes, $\Theta$ is a tensorial matrix defined on $M$ and the curvature $\Omega$ has no $d\xi$ (or $d\bar{\xi}$)-parts. Hence, (\ref{correctitem}) reduces to
\[
\mathfrak{E}=\sum_{j=1}^{n-1}d\text{Re}(\Psi_j)=:d\Xi.
\]
Thus,
\begin{align*}
\int_M {\mathfrak{E}}=\lim_{\epsilon\rightarrow 0^+}\int_{M\epsilon}\mathfrak{E}=-\sum_i\lim_{\epsilon\rightarrow 0^+}\int_{\partial B_{\epsilon}(\zeta_i)}\Xi.
\end{align*}
To show (\ref{newin}), it suffices to prove
\[
\lim_{\epsilon\rightarrow 0^+}\int_{\partial B_{\epsilon}(\zeta_i)}\Xi=0.\tag{5.7}\label{zeron}
\]

Now we define a new Hermitian metric $\widetilde{H}$ on $M$ such that

(i) for each $\zeta_i$, there exists a small neighborhood $U_i$ of $\partial B_\delta(\zeta_i)$ with $\widetilde{H}|_{U_i}=H$,

(ii) $\widetilde{H}|_{B_{\delta/2}(\zeta_i)}=\phi_i^*(H|_{\zeta_i})$, where $\phi_i:T'M|_{B_\delta(\zeta_i)}\rightarrow T_{\zeta_i}M$ is defined as in the proof of Lemma \ref{degreein}, which is the map induced by
 a local trivialization of $T'M|_{B_\delta(\zeta_i)}$.

Let $\widetilde{c_n}(T'M)$ and $\widetilde{\Xi}$ denote the corresponding geometric quantities of $\widetilde{H}$. Since $\widetilde{H}$ is flat on $B_{\delta/2}(\zeta_i)$, $\widetilde{\Xi}=0$. Hence, Property (i) of $\widetilde{H}$ together with (\ref{transsgr}) yields that
\begin{align*}
&-\int_{\partial B_\epsilon(\zeta_i)}\Xi=-\int_{\partial B_\epsilon(\zeta_i)}(\Xi-\widetilde{\,\Xi})=\int_{\partial B_\delta(\zeta_i)}(\Xi-\widetilde{\,\Xi})-\int_{\partial B_\epsilon(\zeta_i)}(\Xi-\widetilde{\,\Xi})\\
=&-\int_{B_\delta(\zeta_i)-B_\epsilon(\zeta_i)}\left[c_n(T'M)-\widetilde{c_n}(T'M)\right]-\int_{\partial B_\delta(\zeta_i)}\text{Re}(\Lambda_2-\widetilde{\Lambda}_2)+\int_{\partial B_\epsilon(\zeta_i)}\text{Re}(\Lambda_2-\widetilde{\Lambda}_2)\\
=&-\int_{B_\delta(\zeta_i)-B_\epsilon(\zeta_i)}\left[c_n(T'M)-\widetilde{c_n}(T'M)\right]+\int_{\partial B_\epsilon(\zeta_i)}\text{Re}({\Lambda}_2-\widetilde{\Lambda}_2)
\end{align*}
Set $H_t:=tH+(1-t)\widetilde{H}$. Clearly, $\frac{d}{dt}H_t|_{\partial B_\delta(\zeta_i)}=0$. Set $L_t:=(\frac{d}{dt}H_t)\cdot H^{-1}_t$. Then the transgression formula \cite[Propositon 3.15]{BC} yields
\[
\int_{B_\delta(\zeta_i)}[c_n(T'M)-\widetilde{c_n}(T'M)]=\int_{\partial B_\delta(\zeta_i)}P'[\Omega_t;L_t]dt=0,
\]
where $P'$ is define as in \cite{BC} and $\Omega_t$ is the curvature of $H_t$. Since $c_n(T'M)$ and $\widetilde{c_n}(T'M)$ are defined on $M$, $\lim_{\epsilon\rightarrow 0}\int_{B_\epsilon(\zeta_i)}[c_n(T'M)-\widetilde{c_n}(T'M)]=0$. Thus, (\ref{zeron}) follows from Lemma \ref{degreein}.
\end{proof}

It is easy to deduce Theorem  \ref{Th1} from Theorem \ref{Gsb}. Let $D$ be the Finslerian connection as in Theorem \ref{Th1}. Thus, \cite[Proposition 2.18]{BC} yields
\[
c_n(\pi^*T'M;\nabla)-c_n(\pi^*T'M;D)=d\mathfrak{G}_F(D),
\]
where $\mathfrak{G}_F(D)$ is a form on $T'M\backslash0$ depending only on $F$ and $D$. Set
\[
\mathfrak{F}_F(X;D):=dX^*\mathfrak{G}_F(D)+\mathfrak{E}.
\]
Theorem \ref{Th1} then follows from Theorem \ref{Gsb}. It is remarkable that $\mathfrak{G}_F(D)$ depends not only on $z$ but also on $\xi$, and therefore, neither $\int_M\frac{dX^*\mathfrak{G}_F(D)}{\vol(z)}$ nor $\int_M\frac{\mathfrak{F}_F(X;D)}{\vol(z)}$ vanishes even if $\vol(z)$ is a constant. For the Hermitian case, $\int_{M}{dX^*\mathfrak{G}_F(D)}=0$ and Theorem \ref{Gsb} furnishes the characteristic class version of (\ref{1.3}).

For some special complex Finsler manifolds, the Gauss-Bonnet-Chern theorem admits significant simplification. For example,
for the complex Berwald $1$-manifolds, we have the following result.
\begin{corollary}\label{Bergs}
Let $M$ be a closed Riemann surface equipped with a complex Berwald metric $F$ and let $X$ be a holomorphic vector field on $M$ with non-degenerate zeros. Then
\[
\frac{\sqrt{-1}}{\vol(z)}\int_{M}\bar{\partial}\partial\log F^2(X)=\chi(M).
\]
\end{corollary}
\begin{proof}
Since $F$ is a Berwald metric, $\vol(z)=\text{const}$. It follows from (\ref{correctitem}) that $\mathfrak{E}=0$ when $n=1$. Hence,
\[
\frac{2\pi}{\vol(z)}\int_M X^*c_1(\pi^*T'M;\nabla)=\chi(M).
\]
We now calculate $X^*c_1(\pi^*T'M)$. On $M_0=M-\cup_i\{\zeta_i\}$, choose a local coordinate system $(z)$ and set $X=f(z)\frac{\partial}{\partial z}$, where $f(z)$ is a holomorphic function. Since $\frac{\partial}{\partial \mathfrak{z}}=\pi^*\frac{\partial}{\partial z}$ is the holomorphic frame field of $\pi^*T'M$, we have
\begin{align*}
&c_1(\pi^*T'M;\nabla)|_{\xi}=\frac{\sqrt{-1}}{2\pi}\bar{\partial}\left[
\frac{\partial g_{\xi}\left(\frac{\partial}{\partial \mathfrak{z}},\frac{\partial}{\partial \mathfrak{z}}\right)}{g_{\xi}\left(\frac{\partial}{\partial \mathfrak{z}},\frac{\partial}{\partial \mathfrak{z}}\right)} \right]=\frac{\sqrt{-1}}{2\pi}\bar{\partial}\partial\left[\log g_{\xi}\left(f(z)\frac{\partial}{\partial \mathfrak{z}},f(z)\frac{\partial}{\partial \mathfrak{z}}\right)\right],
\end{align*}
for all $(z,\xi)\in T'M_0$.
By (\ref{usef}), we have $X^*\circ\partial=\partial \circ X^*$ and therefore,
\[
X^*c_1(\pi^*T'M)=\frac{\sqrt{-1}}{2\pi}\bar{\partial}\partial \log g_{X}(\pi^*X,\pi^*X)=\frac{\sqrt{-1}}{2\pi}\bar{\partial}\partial \log F^2(X).
\]
Then the conclusion follows.
\end{proof}

The Chern form $\frac{\sqrt{-1}}{2\pi}\bar{\partial}\partial\log F^2$ plays an important role in studying Finsler vector bundles. See \cite{A3,A,K} for more details.

\end{document}